\documentclass{amsart}

\usepackage{t1enc}
\usepackage[latin2]{inputenc}
\usepackage{amsmath}
\usepackage{amssymb}
\usepackage{amsthm}
\usepackage{eufrak}
\usepackage{verbatim}
\usepackage{color}
\usepackage{pstricks,pst-node}
\usepackage{enumerate}
\usepackage{xmpmulti}
\usepackage{psfrag}
\usepackage{graphicx}

\newtheorem{thm}{Theorem}[section]
\newtheorem{prop}[thm]{Proposition}
\newtheorem{prob}[thm]{Problem}
\newtheorem{dfn}[thm]{Definition}
\newtheorem{clm}[thm]{Claim}

\newtheorem{cor}[thm]{Corollary}

\newcommand{\rr}{\leq}
\newcommand{\rrr}{\preceq}

\title{The Collins-Roscoe mechanism and D-spaces}
\author{Xu Yuming$^1$}
\author{D\'aniel Soukup$^2$}

\address{$^1$ School of Mathematics, Shandong University,Jinan, China\newline E-mail: xuyuming@sdu.edu.cn }
\address{$^2$ Institute of Mathematics, E\"otv\"os Lor\'and University, Budapest, Hungary\newline E-mail: daniel.t.soukup@gmail.com}
\thanks{Corresponding author. E-mail: daniel.t.soukup@gmail.com \\
This work is supported by NSF of Shandong Province(No.ZR2010AM019) and NNSF of China(No.10971186).}

\keywords{Collins-Roscoe mechanism; well-ordered $(\alpha A)$; linearly semi-stratifiable space; elastic space; $D$-space}
\subjclass[2000]{Primary 54D20. Secondary 54E20}

\begin{document}

\begin{abstract} We prove that if a space  $X$ is well ordered $(\alpha A)$, or linearly semi-stratifiable, or elastic then $X$ is a $D$-space.
\end{abstract}

\maketitle

\section{Introduction}

The connections between D-spaces and generalized metric spaces has been extensively studied. The aim of this paper is to prove the following theorems.
\begin{itemize}
  \item Spaces satisfying well-ordered $(\alpha A)$ are $D$-spaces.
  \item Linearly semi-stratifiable spaces are D-spaces.
  \item Elastic spaces are D-spaces.
\end{itemize}
The proofs are based on Gruenhage's method of \emph{sticky relations}.\\

The paper has the following structure. In Section \ref{secCR} we introduce the Collins-Roscoe mechanism and give the basic definitions. In Section \ref{secsticky} we define the notion of D-spaces and briefly introduce how sticky relations are used to prove that a certain space is D. In Sections \ref{secA}, \ref{secstrat} and \ref{secelastic} we prove the three results above.

\section{The Collins-Roscoe mechanism}\label{secCR}

The expression \emph{Collins-Roscoe structuring mechanism} refers to several definitions of generalized metric properties. In \cite{s1} P. J. Collins and A. W. Roscoe introduced the following notion.

 \begin{dfn}\label{condG}
We say that \emph{a space $X$ satisfies condition (G)} iff there is $\mathcal{W} = \{\mathcal{W}(x): x \in X\}$, where $\mathcal{W}(x)=\{W(m,x):m\in\omega\}$, such that $x\in W(m,x) \subseteq X$ with the following property. For every open set $U$ containing $x \in X$, there exists an open set $V(x, U)$ containing $x$ such that $y \in V(x, U)$ implies $x \in W(m, y) \subseteq U $ for some $m \in \omega$.
\end{dfn}

If we strengthen condition $(G)$ by not allowing the natural number $m$ to vary with $y$, then we say that \emph{$X$ satisfies condition $(A)$}. The precise definition is the following.

 \begin{dfn}We say that \emph{a space $X$ satisfies condition (A)} iff there is $\mathcal{W} = \{\mathcal{W}(x): x \in X\}$, where $\mathcal{W}(x)=\{W(m,x):m\in\omega\}$, such that $x\in W(m,x) \subseteq X$ with the following property. For every open set $U$ containing $x\in X$, there exists an open set $V(x, U)$ containing $x$ and a natural number $m = m(x, U)$ such that
$x \in W(m, y) \subseteq U $ for all $y \in V(x, U)$.
\end{dfn}

If each $W(n, x)$ is open (a neighborhood of $x$), we say that \emph{$X$ satisfies open (neighborhood) $(G)$ or open(neighborhood) $(A)$}, respectively. If $W(n+1, x) \subseteq W(n, x)$ for each $n \in \omega$, we say that \emph{$X$ satisfies decreasing $(G)$ or decreasing $(A)$}.

The Collins-Roscoe mechanism has been extensively studied, and a lot of significant results have been obtained. Let us summarize {\cite[Theorem 1]{s1}} and {\cite[Theorem 8]{s3}} in \ref{metr}.

\begin{thm}\label{metr}The following are equivalent for a space $X$.
\begin{enumerate}[(1)]
  \item $X$ is metrisable,
  \item $X$ satisfies decreasing open $(A)$,
  \item $X$ satisfies decreasing open $(G)$,
  \item $X$ satisfies decreasing neighborhood $(A)$.
\end{enumerate}
\end{thm}

Stratifiable spaces are well known generalizations of metric spaces; see \cite{genmet}. They have a characterization using the Collins-Roscoe mechanism as well. Theorem \ref{strat} summarizes {\cite[Theorem 2.2]{s2}} and a remark from \cite{s3}.

\begin{thm}\label{strat} The following are equivalent for a space $X$.
\begin{enumerate}[(1)]
  \item $X$ is stratifiable,
  \item $X$ satisfies decreasing $(G)$ and has countable pseudo-character,
  \item $X$ satisfies decreasing $(A)$ and has countable pseudo-character.
\end{enumerate}
\end{thm}

We define a third condition denoted by \emph{(F)}, which is weaker than condition (G).

\begin{dfn}We say that \emph{a space $X$ satisfies condition (F)} iff there is $\mathcal{W} = \{\mathcal{W}(x): x \in X\}$ such that $x\in W\subseteq X$ for all $W\in\mathcal{W}(x)$  with the following property. For every open $U$ containing $x$ there is an open $V = V(x, U)$ containing $x$ such that $y \in V$ implies $x \in W \subseteq U$ for some $W \in \mathcal{W}(y)$.
\end{dfn}

We say that \emph{$X$ satisfies well-ordered $(F)$} if each $\mathcal{W}(x)$ is well-ordered by reverse inclusion. We make a remark about well ordered (F) spaces in the last section.

\section{D-spaces and sticky relations}\label{secsticky}

In \cite{s7}, van Douwen and Pfeffer introduced the concept of $D$-spaces.

\begin{dfn}An \emph{open neighborhood assignment (ONA)} on a space $(X, \tau)$
is a function $N : X \rightarrow \tau$ such that $x \in N(x)$ for all $x\in X$. $X$ is a $D$-space iff
for all ONA $N$, there is a closed discrete subset $D$ of $X$ such that
$N[D] = \{N(d): d \in D\}$ covers $X$.
\end{dfn}

We recommend G. Gruenhage's paper \cite{gg} which gives a full review on what we know and do not know about D-spaces.

The following method of G. Gruenhage \cite{gnote} provides us a useful tool for proving that a spaces is a D-space.

\begin{dfn} Let $X$ be a space. A relation $R$ on $X$ is \emph{nearly good} iff $x\in \overline{A}$ implies that there is $y\in A$ such that $xRy$. Let $N$ denote an ONA. If $X'\subseteq X$ and $D\subseteq X$ we say that \emph{$D$ is $N$-sticky mod $R$ on $X'$} if whenever $x \in X'$ and $xRy$ for some $y\in D$ then $x\in \cup N[D]$.
\end{dfn}

\begin{thm}[{\cite[Proposition 2.2]{gnote}}]\label{g1} Let $X$ be a space and $N$ an ONA on $X$. Suppose $R$ is a nearly good relation on $X$ such that every non-empty closed subset $F$ of $X$ contains a non-empty closed discrete subset $D$ which is $N$-sticky
mod $R$ on $F$. Then there is a closed discrete $D^*$ in $X$ with $\cup N[D^*]=X$.
\end{thm}

Let $Z\subseteq X$ and $N$ an ONA on $X$. We say that \emph{$Z$ is $N$-close} iff $Z\subseteq N(x)$ for all $x\in Z$.

\begin{thm}[{\cite[Proposition 2.4]{gnote}}]\label{g2} Let $N$ be a neighborhood assignment on $X$. Suppose there is a
nearly good $R$ on $X$ such that for any $y\in X$, $ R^{-1}(y)\setminus N(y)$ is the countable union of $N$-close sets. Then there is a closed discrete $D$ such that $\cup N[D] = X$.
\end{thm}

 As an easy application of his method, Gruenhage proves in {\cite[Proposition 2.5]{gnote}} that spaces satisfying open (G) are D-spaces. The same proof yields the following.

\begin{prop} If the space $X$ satisfies condition $(G)$ then $X$ is a D-space.
\end{prop}
\begin{proof} Let $\mathcal{W}=\{\mathcal{W}(x):x\in X\}$, where $\mathcal{W}(x)=\{W(n,x):n\in\omega\}$, witness condition (G). We use the notation $V(x,U)$ from Definition \ref{condG} as well.

 Let $N$ be an ONA on $X$. We will apply Theorem \ref{g2} for the following relation $R$. Let $xRy$ iff $x\in W(n,y)\subseteq N(x)$ for some $n\in\omega$. Then $R$ is nearly good; indeed, let $x\in\overline{A}$ for some $A\subseteq X$. Then $V(x, N(x))\cap A\neq\emptyset$ and $xRy$ for any $y\in V(x, N(x))\cap A$.

 Let $y\in X$ and let $C_n=\{x\in X:x\in W(n,y)\subseteq N(x)\}$ for $n\in\omega$. Clearly $R^{-1}(y)=\cup\{C_n:n\in\omega\}$ and $C_n\subseteq W(n,y)$ is $N$-close. Thus by Theorem \ref{g2} there is some closed discrete $D\subseteq X$ such that $X=\cup N[D]$.
\end{proof}

\section{Well ordered $(\alpha A)$ spaces}\label{secA}

Our goal now is to prove that spaces satisfying well-ordered $(\alpha A)$ are D-spaces.

\begin{dfn}\label{alphaA}
Let $X$ be a space, $\alpha$ an ordinal. We say  that \emph{$X$ satisfies $(\alpha A)$} iff there is $\mathcal{W} = \{\mathcal{W}(x): x \in X \}$, where $\mathcal{W}(x)= \{W(\beta, x): \beta < \alpha \}$, such that $x \in W(\beta, x)\subseteq X$ with the following property. For every open $U$ containing $x$, there exists an open set $V(x, U)$ containing $x$ and an ordinal $\beta = \varphi(x, U) < \alpha$ such that
\begin{center}
     $x \in W(\beta, y) \subseteq U $ for all $y \in V(x, U)$.
\end{center}
If, in addition, $W(\beta, x) \subseteq W(\gamma, x)$ whenever $\gamma<\beta<\alpha$, then we say that \emph{$X$ satisfies well-ordered $(\alpha A)$}.
\end{dfn}

\begin{thm} \label{Theorem 2.2} If the space $X$ satisfies  well-ordered $(\alpha A)$ (for some ordinal $\alpha$) then $X$ is a $D$-space.
\end{thm}

  \begin{proof}Let $\mathcal{W}=\{\mathcal{W}(x):x\in X\}$, where $\mathcal{W}(x)=\{W(\beta,x):\beta<\alpha\}$, witness condition $(\alpha A)$. We use the notation $V(x,U)$ and $\varphi(x, U)$ from Definition \ref{alphaA} as well.

   Let $N$ be a neighborhood assignment on $X$. We will define a relation $R$ on $X$ and apply Theorem \ref{g1}. Let $xRy$ iff $x\in W(\beta,y)$ for $\beta=\varphi(x,N(x))$. Clearly $R$ is nearly good; indeed, let $x\in\overline{A}$ for some $A\subseteq X$. Then $V(x,N(x))\cap A\neq\emptyset$ and $xRy$ for any $y\in V(x,N(x))\cap A$.

  Suppose that $F\subseteq X$ is closed and non-empty. We show that there is a closed discrete $D\subseteq F$ such that $D$ is N-sticky mod R on F. Let $\beta_0=\min\{\varphi(y,N(y)):y\in F\}$ and pick $y\in F$ such that  $\beta_0=\varphi(y,N(y))$. Let $D=\{y\}$. Suppose that $xRy$ for some $x\in F$. Then for $\beta=\varphi(x,N(x))$ the following holds $$x\in W(\beta,y)\subseteq W(\beta_0,y)\subseteq N(y)$$
since $\beta\geq\beta_0$. Thus $D$ is $N$-sticky mod $R$ on $F$, and so by Theorem \ref{g1} there is some closed discrete $D^*\subseteq X$ such that $X=\cup N[D^*]$.
\end{proof}

Now we formulate some corollaries. It is proved in \cite{s8} that (semi-)stratifiable spaces are D-spaces. We can slightly strengthen this result.

\begin{dfn}[{\cite[Definition 2.2]{vaugh}}] Let $(X,\tau)$ be a $T_1$ topological space and $\alpha\geq\omega$ an ordinal. $X$ is said to be \emph{stratifiable over $\alpha$} or \emph{linearly stratifiable} iff there exists a mapping $G:\alpha\times \tau\rightarrow \tau$ with the following properties (write $U_\beta=G(\beta, U)$).
\begin{enumerate}
  \item $\overline{U_\beta}\subseteq U$ for all $\beta<\alpha$ and $U\in\tau$,
  \item $\bigcup\{U_\beta:\beta<\alpha\}=U$ for all $U\in\tau$,
  \item if $U\subseteq V$ then $U_\beta\subseteq V_\beta$ for all $\beta<\alpha$,
  \item if $\gamma<\beta<\alpha$ then $U_\gamma\subseteq U_\gamma$ for all $U\in\tau$.
\end{enumerate}
\end{dfn}

From {\cite[Theorem 5.2]{s5}} we know that linearly stratifiable spaces are well-ordered $(\alpha A)$, thus we have the following.

\begin{cor}\label{corlinstrat} Linearly stratifiable spaces are $D$-spaces.
\end{cor}

For a space $(X,\tau)$ let $\mathcal{D}_X=\{(x,U): x\in U\in \tau\}$.

\begin{dfn}A space $(X,\tau)$ is said to be \emph{Borges normal} iff there are operators $H:\mathcal{D}_X\rightarrow \tau$ and $n:\mathcal{D}_X\rightarrow \omega$ such that $H(x, U) \cap H(y, V) \neq \emptyset$ and $n(x, U) \leq n(y, V)$ implies $y \in U$ for all $(x,U),(y,V)\in \mathcal{D}_X$.
\end{dfn}

It can be proved that Borges normal spaces are special well-ordered $(\alpha A)$ spaces.

\begin{thm}[{\cite[Theorem 2.1]{s6}}] A space $X$ is Borges normal iff $X$ satisfies well-ordered $(\omega A)$.
\end{thm}

\begin{cor} Borges normal spaces are $D$-spaces.
\end{cor}

\section{Linearly semi-stratifiable spaces}\label{secstrat}

In \cite{s8}, Borges and Wehrly proved that semi-stratifiable spaces are D-spaces. We find a common generalization of this and Corollary \ref{corlinstrat}, that is, we show that linearly semi-stratifiable spaces are D-spaces.

 Let $(X, \tau)$ be a $T_{1}$-space and let $\mathcal{F}_X$ denote the family of all closed subsets of $X$.

\begin{dfn} $X$ is said to be \emph{semi-stratifiable over $\alpha$} (for some ordinal $\alpha$) or \emph{linearly semi-stratifiable} if there exists a mapping $F: \alpha \times \tau \rightarrow \mathcal{F}_X$ such that:
\begin{enumerate}[(1)]
\item $U = \cup \{F(U, \beta): \beta < \alpha \}$ for all $U \in \tau$;

\item if $U \subseteq W$ then $F(U, \beta) \subseteq F(W, \beta)$ for all $\beta < \alpha$;

\item if $\gamma < \beta < \alpha$, then $F(U, \gamma) \subseteq F(U, \beta)$ for all $U \in \tau$.
\end{enumerate}
\end{dfn}

\begin{thm} If the space $X$ is semi-stratifiable over $\alpha$ (for some ordinal $\alpha$) then $X$ is a $D$-space.
\end{thm}
 \begin{proof} Let $F:\alpha \times \tau \rightarrow \mathcal{F}_X$ be the function witnessing that $X$ is linearly semi-stratifiable.

  Let $N$ be ONA on $X$. We will define a relation $R$ on $X$ and apply Theorem \ref{g1}. Let $\sigma(x)=\min\{\beta<\alpha:x\in F(N(x),\beta)\}$ for $x\in X$. Let $xRy$ iff $x\in N(y)$ or $\sigma(x)<\sigma(y)$. We prove that $R$ is nearly good. Suppose that $x\in\overline{A}$ however $x\notin R^{-1}(y)$ for all $y\in A$. Thus $x\notin \cup\{N(y):y\in A\}$ and $\sigma(y)\leq\sigma(x)$ for all $y\in A$. Thus $y\in F(N(y),\sigma(y))\subseteq F(N(y),\sigma(x))$ for all $y\in A$. Thus $$A\subseteq F(\cup\{N(y):y\in A\},\sigma(x))\subseteq \cup\{N(y):y\in A\} \subseteq X\setminus\{x\}.$$ $F(\cup\{N(y):y\in A\},\sigma(x))$ is closed hence $x\in \overline{A}\subseteq F(\cup\{N(y):y\in A\},\sigma(x))$, which is a contradiction. This proves that $R$ is nearly good.

Suppose that $F\subseteq X$ is closed and nonempty. We show that there is a closed discrete $D\subseteq F$ such that $D$ is N-sticky mod R on F. Let $\sigma=\min\{\sigma(y):y\in F\}$ and let $y\in F$ such that $\sigma=\sigma(y)$. Let $D=\{y\}$. If $xRy$ for some $x\in F$ then $x\in N(y)$ since $\sigma(x)\geq\sigma(y)$. Thus $D$ is $N$-sticky mod $R$ on $F$, and so by Theorem \ref{g1} there is some closed discrete $D^*\subseteq X$ such that $X=\cup N[D^*]$.
\end{proof}

\section{Elastic spaces}\label{secelastic}

Our aim now is to prove that \emph{elastic spaces} are D-spaces. Elastic spaces were first introduced by H. Tamano and J. E. Vaughan in \cite{s16} as a natural generalization of stratifiable spaces. First we need the definition of a \emph{pair-base} which is due to J. G. Ceder \cite{s15}.

\begin{dfn}Let $X$ be a space. A collection $\mathcal{P}$ of ordered pairs $P = (P_{1}, P_{2})$ of subsets of  $X$ is called a \emph{pair-base} provided that $P_{1}$ is open for all $P \in \mathcal{P}$ and that for every $x \in X$ and open set $U$ containing $x$, there is a $P \in \mathcal{P}$ such that $x \in P_{1} \subseteq P_{2} \subseteq U$.
\end{dfn}

 The following definition of elastic spaces is an improvement of the original one and due to Gartside and Moody \cite{s4}.

\begin{dfn} A space $X$ is \emph{elastic} if there is a pair-base $\mathcal{P}$ on $X$
and transitive relation $\leq$ on $\mathcal{P}$ such that
\begin{enumerate}[(1)]
 \item if $P, P^{\prime} \in \mathcal{P}$ are such that $P_{1} \cap P_{1}^{\prime} \neq \emptyset$ then $P \leq P^{\prime}$ or $P^{\prime} \leq P$;

 \item if $P \in \mathcal{P}$ and $\mathcal{P}^{\prime} \subseteq \{P^{\prime} \in \mathcal{P}: P^{\prime} \leq P\}$ then $\overline{\cup \{P_{1}^{\prime}: P^{\prime} \in \mathcal{P}'\}} \subseteq \cup \{P_{2}^{\prime}: P^{\prime} \in \mathcal{P}'\}$.
\end{enumerate}
\end{dfn}

Note that the relation $\leq$ should be reflexive.

Before we show that elastic spaces are D-spaces, we need the following proposition which is implicitly in {\cite[Lemma 2]{s16}}.

\begin{prop}\label{2.11} Suppose that $\rr$ is a reflexive, transitive relation on the set $S$, then there is a reflexive, antisymmetric relation $\rrr$ on $S$ such that:
\begin{enumerate}[(1)]
\item if $x, y \in S$ and $x \rr y$, then $x \rrr y$ or $y \rrr x$;

\item if $A$ is a non-empty subset of $S$, then A has a $\rrr$-minimal element, i.e. there is an $x \in A$ such that $y \not\rrr x$ whenever $y \in A \setminus \{x\}$;

\item if $A \subseteq S$ and $A \subseteq \{x \in S: x \rrr s \}$ for some $s \in S$,
then $A \subseteq \{x \in S: x \rr s'\}$ for some $s'\in S$.
\end{enumerate}
\end{prop}
\begin{proof} Let $S=\{s_\alpha:\alpha<\kappa\}$ and $S(\alpha)=\{s\in S: s\rr s_\alpha\}$ for $\alpha<\kappa$. By induction on $\alpha<\kappa$ we define a reflexive and antisymmetric relation $\rrr_\alpha$ on $\cup\{S(\beta):\beta\leq \alpha\}$ such that $\rrr_\alpha$ extends $\rrr_\beta$ for $\beta<\alpha<\kappa$ and then let $\rrr$ to  be $\cup\{\rrr_\alpha:\alpha<\kappa\}$.

 Let $\rrr_0$ denote a well-ordering of $S(0)$. Let $\alpha<\kappa$ and suppose that $\rrr_\beta$ is constructed for $\beta<\alpha$. Let $S'(\alpha)=S(\alpha)\setminus \cup\{S(\beta:\beta<\alpha)\}$. Let $\rr_\alpha$ be a well-ordering on $S'(\alpha)$ and also put $s\rr_\alpha s'$ if $s'\in S'(\alpha)$ and $s\in S(\alpha)\cap(\cup\{S(\beta):\beta<\alpha\})$. Let $\rrr_\alpha=\cup\{\rrr_\beta:\beta<\alpha\}\cup\rr_\alpha$; this is reflexive and antisymmetric. Finally, let $\rrr$ to be $\cup\{\rrr_\alpha:\alpha<\kappa\}$.

Clearly $\rrr$ is reflexive and antisymmetric on $S$. First we shall verify (1). Let us suppose that $x\rr y$ for some $x,y\in S$. Let $\alpha_0=\min\{\alpha<\kappa:y\in S(\alpha)\}$, since $\rr$ is transitive we have $x\in S(\alpha_0)$. Then by definition $x\rr_\alpha y$ or $y\rr_\alpha x$ thus $x\rrr y$ or $y\rrr x$.

Next we show that every nonempty $A\subseteq S$ has a $\rrr$-minimal element. First note the following.

\begin{clm} If $s,s'\in S$, $s\rrr s'$ and $\alpha<\kappa$ is minimal such that $s,s'\in \cup\{S(\beta):\beta\leq \alpha\}$ then $s\rrr_\alpha s'$, $s'\in S'(\alpha)$ and $s\in S(\alpha)$.
\end{clm}
\begin{proof} Let $\gamma<\kappa$ minimal such that $s\rrr_\gamma s'$. Thus $s,s'\in \cup\{S(\beta):\beta\leq \gamma\}$ hence $\alpha\leq \gamma$. If $\alpha<\gamma$ then $s,s'\notin S'(\gamma)$ so $s$ and $s'$ are not related by $\rr_\gamma$. Hence there is some $\beta<\gamma$ such that $s\rrr_\beta s'$ (by the definition of $\rrr_\gamma$). This contradicts the choice of $\gamma$. Thus $\alpha=\gamma$ and $s\rrr_\alpha s'$. Clearly $s\rr_\alpha s'$ by the definition of $\rrr_\alpha$ since $s$ and $s'$ are not related by $\rrr_\delta$ for any $\delta<\alpha=\gamma$ . Thus $s'\in S'(\alpha)$ and $s\in S(\alpha)$.
\end{proof}

Suppose that $\emptyset \neq A\subseteq S$. Let $\alpha_0=\min\{\alpha<\kappa:A\cap S(\alpha)\neq\emptyset\}$. Then $A\cap S(\alpha_0)\subseteq S'(\alpha_0)$. Since $S'(\alpha_0)$ is well-ordered by $\rr_{\alpha_0}$, there is an $x\in A\cap S(\alpha_0)$ which is $\rr_{\alpha_0}$-minimal in $A\cap S(\alpha_0)$. We show that $x$ is $\rrr$-minimal in $A$. Clearly $x$ is $\rrr$-minimal in $A\cap S(\alpha_0)$. If $y\rrr x$ for some $y\in A$ then for the minimal $\alpha<\kappa$ such that $x,y\in \cup\{S(\beta):\beta\leq \alpha\}$ we have $\alpha_0<\alpha$. By the claim $x\in S'(\alpha)$ which is a contradiction. Thus $x$ is $\rrr$-minimal in $A$, i.e. (2) holds.

Finally we show that if $A$ is $\rrr$ upper bounded then also $\rr$ upper bounded. Suppose that $A \subseteq \{x \in S: x \rrr s \}$ for some $s \in S$. Let $\alpha_0=\min\{\alpha<\kappa:s\in S(\alpha)\}$. We shall show that $A \subseteq S(\alpha_0)$, that is, $s_{\alpha_0}$ is a $\rr$ upper bound for $A$. Clearly $s\in S'(\alpha_0)$. Let $x\in A$ and let $\alpha$ be minimal such that $x,s\in \cup\{S(\beta):\beta\leq \alpha\}$. Then $s\in S'(\alpha)$ by the claim and $x\rrr s$. Hence $\alpha=\alpha_0$ and $x\in S(\alpha_0)$, using the claim again. This proves $A\subseteq S(\alpha_0)$.
\end{proof}

\begin{thm} If $X$ is elastic then $X$ is a D-space.
\end{thm}
\begin{proof} Let $\mathcal{P}$ be the pair-base on $X$ with some relation $\rr$ witnessing that $X$ is elastic. There is a reflexive antisymmetric relation $\rrr$ on $\mathcal{P}$ by Proposition \ref{2.11} with the following properties:
\begin{enumerate}[(a)]
  \item if $P, P^{\prime} \in \mathcal{P}$ are such that $P_{1} \cap P_{1}^{\prime} \neq \emptyset$ then $P \rrr P^{\prime}$ or $P^{\prime} \rrr P$;
  \item if $\mathcal{P}'$ is a non-empty subset of $\mathcal{P}$, then there is a $\rrr$-minimal element of $\mathcal{P}'$;
  \item if $P \in \mathcal{P}$ and $\mathcal{P}^{\prime} \subseteq \{P^{\prime} \in \mathcal{P}: P^{\prime} \rrr P\}$ then $\overline{\cup \{P_{1}^{\prime}: P^{\prime} \in \mathcal{P}'\}} \subseteq \cup \{P_{2}^{\prime}: P^{\prime} \in \mathcal{P}'\}$.
\end{enumerate}

Let us enumerate $\mathcal{P}$ as follows. By property (b), there is an element of $\mathcal{P}$, denoted by $P^{0}$, such that $P \not\rrr P^{0}$ whenever
$P \in \mathcal{P} \setminus \{P^{0}\}$. Assume $P^{\gamma}$ has been selected for each $\gamma < \beta$,
and $P \not\rrr P^{\gamma}$ whenever $P \in \mathcal{P} \setminus \{P^{\eta}: \eta \leq \gamma\}$.
If $\mathcal{P} \setminus \{P^{\gamma}: \gamma < \beta\} \neq \emptyset$, there is an element of
$\mathcal{P} \setminus \{P^{\gamma}: \gamma < \beta\}$, denoted by $P^{\beta}$, such that
$P \not\rrr P^{\beta}$ whenever $P \in \mathcal{P} \setminus \{P^{\gamma}: \gamma \leq \beta\}$.
Thus $\mathcal{P}$ can be enumerated as $\mathcal{P} = \{P^{\beta}: \beta < \lambda \}$ such that
\begin{enumerate}[(d)]
\item $P^{\beta'} \not \rrr  P^{\beta}$ if $\beta < \beta' < \lambda$.
\end{enumerate}
Let $N$ be an ONA on $X$. We will define a relation $R$ on $X$ and apply Theorem \ref{g1}. Let $\sigma(x) = \min \{ \beta < \lambda: x \in P^{\beta}_1 \subseteq P^{\beta}_2\subseteq N(x)\}$ for $x\in X$. Let $xRy$ iff $x\in N(y)$ or $P^{\sigma(x)} \rrr P^{\sigma(y)}$. We prove that $R$ is nearly good. Suppose that $x\in\overline{A}$ however $x\notin R^{-1}(y)$ for all $y\in A$. Thus $x\notin N(y)$ and $P^{\sigma(x)}\not \rrr P^{\sigma(y)}$ for all $y\in A$. Let $A_1 = A \cap P^{\sigma(x)}_1\neq \emptyset$. Since $P^{\sigma(x)}_1 \cap P^{\sigma(y)}_1 \neq \emptyset$ we have $P^{\sigma(y)} \rrr P^{\sigma(x)}$ for all $y \in A_1$. Thus $$ \overline {\cup \{P^{\sigma(y)}_1: y \in A_1\}}\subseteq \cup \{P^{\sigma(y)}_2: y \in A_{1}\} \subseteq \cup \{N(y): y \in A_{1}\} \subseteq X\setminus\{x\}$$

using $P^{\sigma(y)}_2 \subseteq N(y)$ and $x \notin N(y)$ for $y\in A_1$. Clearly $x \in \overline{A_1}$ and $$A_1 \subseteq \cup \{P^{\sigma(y)}_1: y \in A_1\}.$$

This yields $x \in \overline{A_1} \subseteq X \setminus \{x\}$ which is a contradiction. Thus $R$ is nearly good.

Suppose that $F \subseteq X$ is closed and nonempty. We show that there is a closed discrete $D\subseteq F$ such that $D$ is N-sticky mod R on F. Let $\sigma= \min \{\sigma(y): y \in F\}$ and let $y \in F$ such that $\sigma = \sigma(y)$. Let $D = \{y\}$. Suppose $x R y$ for some $x \in F$. If $P^{\sigma(x)} \rrr P^{\sigma(y)}$ then $\sigma(x)=\sigma(y)$ since $\sigma(y)\leq\sigma(x)$ (and by property (d)). Thus $x\in P^{\sigma(x)}_1=P^{\sigma(y)}_1\subseteq P^{\sigma(y)}_2\subseteq N(y)$. If  $P^{\sigma(x)} \not\rrr P^{\sigma(y)}$ then $x \in N(y)$. Thus $D$ is $N$-sticky mod $R$ on $F$, and so by Theorem \ref{g1} there is some closed discrete $D^* \subseteq X$ such that $X = \cup N[D^*]$.
\end{proof}

\emph{Proto-metrisable} spaces were introduced by P. Nyikos in his study of nonarchimedian spaces in \cite{s17}.

 \begin{dfn} Let $X$ be a space, $\mathcal{B}$ a base for the topology. The base $\mathcal{B}$ is said to be an \emph{orthobase} if whenever $\mathcal{B}' \subseteq \mathcal{B}$, either $\cap \mathcal{B}'$ is open or $\mathcal{B}'$ is a local base for any point in $\cap \mathcal{B}'$. A space is said to be \emph{proto-metrisable} if it is paracompact and has an orthobase.
 \end{dfn}

 Gartside and Moody proved that proto-metrisable spaces are elastic {\cite[Corollary 9]{s18}}. Thus we can deduce the following corollary, which had already been obtained by Borges and Wehrly in \cite{s9}.

\begin{cor} Every proto-metrisable space is a $D$-space.
\end{cor}

Finally, let us mention a long standing problem of Borges and Wehrly. In \cite{s8}, the authors asked whether monotonically normal paracompact spaces are D-spaces. Almost twenty years past, this question remains open. The following implications can be proved; for details see \cite{stratel}, \cite{s4} and \cite{s3}.

$$\text{metrisable}\Rightarrow \text{(linearly-)stratifiable} \Rightarrow \text{\textbf{elastic}} \Rightarrow$$ $$\Rightarrow \text{well-ordered (F)} \Rightarrow \text{monotone normal and (her.) paracompact}$$

 Since we know that elastic spaces are D-spaces, we think the following question is valuable to study.

\begin{prob} Are well-ordered $(F)$ spaces $D$-spaces?
\end{prob}

We mention that Y. Z. Gao, H. Z. Qu and S. T.Wang gave an interesting characterization for monotonically normal paracompact spaces in \cite{s19}.

\end{document}